\documentclass[11pt,a4paper]{article}
\usepackage[T1]{fontenc}
\usepackage{iftex}
\ifPDFTeX
\usepackage[utf8]{inputenc}
\fi
\usepackage{lmodern}
\usepackage{amsmath,amssymb,amsthm,mathtools}
\usepackage[margin=27mm,headheight=15pt]{geometry}
\usepackage{enumitem}
\usepackage{booktabs,array}
\usepackage{microtype}
\usepackage{xcolor}
\usepackage{fancyhdr}
\usepackage[colorlinks=true,linkcolor=black,citecolor=black,urlcolor=blue!45!black]{hyperref}
\hypersetup{pdftitle={Reduction of integer tiles via CRT and base-p digits},pdfauthor={Chunlin Li, Jie Wu, Wenchuan Hu, Chao Wang, Erxiao Wang, Fuhai Zhu}}
\setlist[enumerate]{label=\textup{(\roman*)},itemsep=.3em,topsep=.4em,leftmargin=2em}
\theoremstyle{plain}
\newtheorem{theorem}{Theorem}[section]
\newtheorem{lemma}[theorem]{Lemma}
\newtheorem{proposition}[theorem]{Proposition}
\newtheorem{corollary}[theorem]{Corollary}
\theoremstyle{remark}

\numberwithin{equation}{section}
\newcommand{\Z}{\mathbb Z}
\newcommand{\Q}{\mathbb Q}
\newcommand{\R}{\mathbb R}
\newcommand{\C}{\mathbb C}
\newcommand{\F}{\mathbb F}
\newcommand{\ord}{\operatorname{ord}}

\newcommand{\eps}{\varepsilon}
\makeatletter
\let\orig@maketitle\maketitle
\renewcommand{\maketitle}{%
	\begingroup
	\let\@fnsymbol\@arabic
	\def\thanks##1{\footnotemark}
	\orig@maketitle
	\endgroup
	
	\begingroup
	\long\def\@makefntext##1{\parindent 1em\noindent ##1}
	\footnotetext[1]{%
		\begin{tabular}{@{}p{0.49\textwidth}@{\hspace{0.02\textwidth}}p{0.49\textwidth}@{}}
			\textsuperscript{\normalfont 1}chunlinli@zjnu.edu.cn, Zhejiang Normal University &
			\textsuperscript{\normalfont 2}wwwwwj@zjnu.edu.cn, Zhejiang Normal University \\[2pt]
			\textsuperscript{\normalfont 3}wenchuan@scu.edu.cn, Sichuan University &
			\textsuperscript{\normalfont 4}wangchao@nankai.edu.cn, Nankai University \\
		\end{tabular}
		\par\vspace{2pt}
		\noindent\textsuperscript{\normalfont 5}Corresponding author (wang.eric@zjnu.edu.cn, Zhejiang Normal University). Research was supported by National Natural Science Foundation of China NSFC-RGC 12361161603.
		\par\vspace{2pt}
		\noindent\textsuperscript{\normalfont 6}zhufuhai@nju.edu.cn, Nanjing University
	}%
	\endgroup
}
\makeatother

\title{Reduction of integer tiles via CRT and base-$p$ digits}

\author{Chunlin Li\thanks{chunlinli@zjnu.edu.cn, Zhejiang Normal University}, \enspace
	Jie Wu\thanks{wwwwwj@zjnu.edu.cn, Zhejiang Normal University}, \enspace
	Wenchuan Hu\thanks{wenchuan@scu.edu.cn, Sichuan University},  \enspace
	Chao Wang\thanks{wangchao@nankai.edu.cn, Nankai University}, \enspace
	Erxiao Wang\thanks{Corresponding author (wang.eric@zjnu.edu.cn, Zhejiang Normal University).  Research was supported by National Natural Science Foundation of China NSFC-RGC 12361161603.}, \enspace
	Fuhai Zhu\thanks{zhufuhai@nju.edu.cn, Nanjing University}}
\date{}

\begin{document}
\maketitle
\thispagestyle{plain}

\begin{abstract}
Coven and Meyerowitz gave two cyclotomic conditions, (T1) and (T2),
which characterize integer tiles whose cardinalities have at most two
distinct prime factors. We prove that the same characterization holds
without this restriction. The proof uses a reduction in Chinese remainder
coordinates: slicing by the lowest base-$p$ digit produces sets with a
common tiling complement in a group of order smaller by a factor of $p$.
This reduction preserves the cyclotomic data needed for an induction on
the exponents in (T2).
\end{abstract}



\section{Introduction}

A finite nonempty set $A\subset\Z$ is an \emph{integer tile} if there is
a set $T\subset\Z$ such that every integer has a unique representation
$a+t$, with $a\in A$ and $t\in T$. We write $A\oplus T=\Z$.
After a translation, we may assume that $\min A=0$ and form the mask
polynomial
\[
 A(x)=\sum_{a\in A}x^a.
\]
Translation does not affect its cyclotomic divisors. Let $\Phi_s(x)$ be
the $s$th cyclotomic polynomial, write $\Phi_s\mid A$ for
$\Phi_s(x)\mid A(x)$, and put
\[
 S_A=\{p^a:p\text{ prime},\ a\ge1,\ \Phi_{p^a}\mid A\}.
\]
The Coven--Meyerowitz conditions are
\begin{align*}
 \textup{(T1)}\quad &|A|=\prod_{s\in S_A}\Phi_s(1),\\
 \textup{(T2)}\quad &s_1,\ldots,s_k\in S_A
   \text{ are powers of distinct primes}
   \ \Longrightarrow\ \Phi_{s_1\cdots s_k}\mid A.
\end{align*}

Newman~\cite{Newman} characterized integer tiles of prime-power
cardinality. In the notation above, his criterion is (T1); under (T1),
condition (T2) is automatic in this case. Coven and
Meyerowitz~\cite[Theorems A, B1, and B2]{CM} proved that (T1) and (T2)
are sufficient for tiling, that (T1) is necessary, and that (T2) is
necessary when $|A|$ has at most two distinct prime divisors.

Every tiling of $\Z$ by translates of a finite set is periodic
\cite[Lemma~1.2]{CM}. Thus, if $A\oplus T=\Z$, there are an
integer $M\ge1$ and a set $B\subset\{0,\ldots,M-1\}$ such that
$T=B+M\Z$ and
\[
A\oplus B=\Z_M,\qquad |A||B|=M,
\]
where $A$ is identified with its image modulo $M$.
Conversely, every such cyclic factorization gives an integer
tiling $A\oplus(B+M\Z)=\Z$. This correspondence allows us to
formulate tiling results in terms of the period $M$.

The subgroup reduction used in the two-prime case does not extend
directly to three primes. \L{}aba and Londner~\cite{LL0} proved that
\[
 A\oplus B=\Z_{p^2q^2r^2},\qquad |A|=|B|=pqr,
\]
implies (T2) for both factors when $p,q,r$ are distinct odd primes;
they later removed the parity restriction~\cite{LL1}.
Together with the reduction for factorizations whose cardinalities
share at most two prime divisors, this establishes (T2) for all tilings
of $\Z_{p^2q^2r^2}$. This is a restriction on the period, rather than
a result for all tiles of cardinality $pqr$.
Their splitting method also gives (T2) for every tiling with either of
the periods
\[
\begin{aligned}
 M&=p_1^{n_1}p_2^{n_2}p_3^{n_3},
 &p_1&>p_2^{n_2-1}p_3^{n_3-1},\\
 M&=p_1^{n_1}p_2^2p_3^2p_4^2,
 &p_1&>p_2p_3p_4,
\end{aligned}
\]
where the primes are distinct and the exponents are
positive~\cite{LL}. Counterexamples to functional versions of these
conditions involve nonnegative functions, rather than set
masks~\cite{KLM}.

We prove the necessity of (T2) for arbitrary finite integer tiles.

\begin{theorem}\label{thm:main}
A finite nonempty set $A\subset\Z$ tiles $\Z$ by translations if and
only if its mask polynomial satisfies \textup{(T1)} and \textup{(T2)}.
\end{theorem}

The main step is a reduction that generalizes the subgroup reduction 
in Lemma 2.5 of \cite{CM}. Write a cyclic period as $np^K$, with $(n,p)=1$, and 
use CRT-coordinates $\Z_n\times\Z_{p^K}$. If $\Phi_p$ divides one factor, 
its slices by the lowest base-$p$ digit have a common complement in
$\Z_n\times\Z_{p^{K-1}}$ (Theorem~\ref{thm:carrier}).
The construction uses the carry in digit addition and the fact that
the coefficients of set masks are zero or one. We then choose a slice
or a complement that preserves a hypothetical failure of (T2), while
decreasing the sum of the selected exponents. Section~\ref{sec:example}
examines both choices for an explicit tiling of \(\mathbb{Z}_{72}\), which may 
help readers become familiar with the reduction procedure and the main constructions used in the proofs.

In retrospect, the significance of base-$p$ digit expansions was already suggested in Newman’s  paper \cite{Newman} in 1977. He remarked that the case of cardinality $6$ probably ``depends simultaneously on the 2-adic and 3-adic expansions'' in the first page. 

There is also a consequence for Fuglede's spectral set
conjecture~\cite{Fuglede}. A finite subset of an abelian group is
\emph{spectral} if the restrictions of some characters form an
orthogonal basis for the functions on that subset. For a measurable
set $F\subset\R$ of finite positive measure, spectrality means that
$\mathrm{L}^2(F)$ has an orthogonal basis of exponential functions.
Combining Theorem~\ref{thm:main} with \L{}aba's spectral construction 
 \cite{Laba} and Dutkay-Lai's \cite[Theorems~1.8(ii)]{DL} gives  
the following ``tiling implies spectrality'' part of Fuglede conjecture. 

\begin{corollary}\label{cor:spectral}
\begin{enumerate}
 \item Every finite tile of $\Z_M$ or $\Z$ is spectral.
 \item Every bounded measurable tile  of $\R$ with finite positive measure is spectral. 
\end{enumerate}
\end{corollary}

The converse remains open.

\section{Cyclic tilings and algebraic preliminaries}

\subsection{Cyclic realization and cyclotomic factors}

For each positive integer $m$, write $\Z_m=\Z/m\Z$ and
\[
H_m(x)=1+x+\cdots+x^{m-1}.
\]
Masks of subsets of $\Z_M$ are formed using the standard
representatives $0,\ldots,M-1$. For $E,F\subset\Z_M$, the
factorization $E\oplus F=\Z_M$ is equivalent to
\begin{equation}\label{eq:cyclic}
	E(x)F(x)=H_M(x)
	\quad\text{in }\Z[x]/(x^M-1),
\end{equation}
by~\cite[Lemma~1.3]{CM}.

The following formulation follows from the periodicity
theorem~\cite[Lemma~1.2]{CM} by taking a sufficiently large
multiple of a tiling period.

\begin{lemma}[Cyclic realization]\label{lem:period}
	Let $A\subset\Z$ be a finite integer tile. Given positive
	integers $d_1,\ldots,d_t$, there is a common multiple $M$
	of the $d_i$ such that reduction modulo $M$ is injective
	on $A$ and, identifying $A$ with its image,
	\[
	A\oplus B=\Z_M
	\]
	for some $B\subset\Z_M$.
\end{lemma}

We also recall the prime-power allocation result of
Coven and Meyerowitz~\cite[Lemma~2.1]{CM}.

\begin{lemma}[Prime-power allocation]\label{lem:census}
	Suppose that $E\oplus F=\Z_M$. Then $S_E$ and $S_F$
	are disjoint, and their union is the set of all
	prime-power divisors of $M$. For each prime $p\mid M$,
	exactly $v_p(|E|) := \mathrm{max} \{ j \in \Z_{\ge 0} : ~ p^j ~ | ~ |E| \} $ of the polynomials
	\[
	\Phi_p,\Phi_{p^2},\ldots,\Phi_{p^{v_p(M)}}
	\]
	divide $E(x)$, and the remaining $v_p(|F|)$ divide $F(x)$.
\end{lemma}

\subsection{Chinese remainder coordinates}

Let $M=np^K$, where $p$ is prime and $(n,p)=1$. Fix
\begin{equation}\label{eq:crt}
 \kappa_K:\Z_M\longrightarrow\Z_n\times\Z_{p^K},
 \qquad a\longmapsto(a\bmod n,a\bmod p^K).
\end{equation}
For $E\subset\Z_n\times\Z_{p^K}$, put
\[
 E(x,t)=\sum_{(h,k)\in E}x^ht^k,\qquad
 E^\flat=\kappa_K^{-1}(E).
\]
Both coordinates in the sum use standard representatives. If
$\xi^n=\lambda^{p^K}=1$, then
\begin{equation}\label{eq:dictionary}
 E(\xi,\lambda)=E^\flat(\xi\lambda).
\end{equation}
Indeed, the two sides agree on each monomial. If $\ord\xi=d$ and
$\ord\lambda=p^a$, their product has order $dp^a$, and hence
\begin{equation}\label{eq:cyclodict}
 E(\xi,\lambda)=0
 \quad\Longleftrightarrow\quad \Phi_{dp^a}\mid E^\flat.
\end{equation}
Here $a=0$ is allowed. This is the minimal-polynomial criterion for
a primitive root of unity. We use the same notation with $K-1$
in place of $K$, including $K-1=0$.

\begin{lemma}[Evaluation at roots of unity]\label{lem:interp}
Let $F\in\C[x,y]/(x^n-1,y^L-1)$ for positive integers $n,L$. If
$F(\xi,\zeta)=0$ whenever $\xi^n=\zeta^L=1$, then $F=0$.
\end{lemma}
\begin{proof}
Use the representative with $\deg_x F<n$ and $\deg_y F<L$.
For each $\zeta^L=1$, the polynomial $F(x,\zeta)$ has $n$
distinct zeros and degree less than $n$, so it is zero.
Each coefficient, viewed as a polynomial in $y$, then has $L$
distinct zeros and degree less than $L$, and is also zero.
\end{proof}

\medskip

Let $\varphi(m)$ denote the Euler function. Let $\zeta_p$ denote the primitive $p$-th root of unity.
\begin{lemma}[Cyclotomic extensions]\label{lem:field}
Let $\xi$ have order $m$, with $(m,p)=1$.
\begin{enumerate}
 \item If $\ord\zeta=p^v$ and $v\ge1$, then $t^p-\zeta$
 is irreducible over $\Q(\xi,\zeta)$.
 \item The polynomial $\Phi_p(t)$ is irreducible over $\Q(\xi)$.
\end{enumerate}
\end{lemma}
\begin{proof}
A root $\rho$ of $t^p-\zeta$ has order $p^{v+1}$. The
cyclotomic degree formula~\cite[Lemma~5.9 and Theorem~5.10]{Milne} gives
\[
 [\Q(\xi,\rho):\Q(\xi,\zeta)]
 =\frac{\varphi(mp^{v+1})}{\varphi(mp^v)}=p,
\]
so $t^p-\zeta$ is its minimal polynomial. Likewise,
$[\Q(\xi,\zeta_p):\Q(\xi)]=\varphi(mp)/\varphi(m)=p-1$,
which proves the second assertion.
\end{proof}

\medskip

 Let $\mathbb{F}_p := \mathbb{Z}/p\mathbb{Z}$ denote the finite field with $p$ elements.

\begin{lemma}\label{lem:reduced}
If $(n,p)=1$, the ring $\F_p[x]/(x^n-1)$ has no nonzero
nilpotent elements.
\end{lemma}
\begin{proof}
In this ring,
\[
 \left(\sum_{h=0}^{n-1}a_hx^h\right)^p
 =\sum_{h=0}^{n-1}a_hx^{ph}.
\]
Multiplication by $p$ permutes $\Z_n$, so the Frobenius map
is injective. Its iterates are injective as well, excluding
nonzero nilpotents.
\end{proof}

\section{Digit slices and a common complement}

Fix a prime $p$ and positive integers $n,K$ with $(n,p)=1$, and write
\[
 L=p^{K-1},\qquad
 \widetilde G=\Z_n\times\Z_{p^K},\qquad
 G=\Z_n\times\Z_L.
\]
Throughout this section, suppose that $P\oplus Q=\widetilde G$ and
\begin{equation}\label{eq:lowowner}
 Q(1,\zeta_p)=0,
\end{equation}
where $\zeta_p$ is a primitive $p$th root of unity.

\subsection{Products of slices}

Writing the second coordinate as $r+pj$, define
\begin{align}
 E_r&=\{(h,j)\in G:(h,r+pj)\in P\},\label{eq:Eslice}\\
 C_q&=\{(h,j)\in G:(h,q+pj)\in Q\}.\label{eq:Cslice}
\end{align}
The indices $r,q$ range over $0,\ldots,p-1$. In
$R=\Z[x,y]/(x^n-1,y^L-1)$, set
\[
 U_{r,q}=E_rC_q,\qquad W=H_n(x)H_L(y).
\]
For $\eps(r,q)=\lfloor(r+q)/p\rfloor$, the tiling identity becomes
\begin{equation}\label{eq:carry}
 \sum_{\substack{0\le r,q<p\\r+q\equiv c\pmod p}}
 y^{\eps(r,q)}U_{r,q}=W\qquad(0\le c<p).
\end{equation}
The power of $y$ records the carry from the lowest digit.
Also,~\eqref{eq:lowowner} implies that
$\sum_q|C_q|t^q$ is a multiple of $\Phi_p(t)$, and thus
\begin{equation}\label{eq:equalmass}
 |C_0|=\cdots=|C_{p-1}|=:b>0.
\end{equation}

\begin{lemma}\label{lem:rigid}
Every standard coefficient of $U_{r,q}$ is zero or one, and
\begin{align}
 (y-1)U_{r,q}&=0,\label{eq:periodic}\\
 U_{r,q}+U_{r',q'}-U_{r,q'}-U_{r',q}&=0\label{eq:rectangle}
\end{align}
for all indices.
\end{lemma}
\begin{proof}
Fix roots of unity $\xi,\zeta$ satisfying $\xi^n=\zeta^L=1$, and put
\[
 a_r=E_r(\xi,\zeta),\quad b_q=C_q(\xi,\zeta),\qquad
 A_{\xi,\zeta}(t)=\sum_ra_rt^r,\quad
 B_{\xi,\zeta}(t)=\sum_qb_qt^q.
\]
If $(\xi,\zeta)\ne(1,1)$, evaluating~\eqref{eq:carry} gives
\begin{equation}\label{eq:productquotient}
 A_{\xi,\zeta}(t)B_{\xi,\zeta}(t)=0
 \quad\text{in }\Q(\xi,\zeta)[t]/(t^p-\zeta).
\end{equation}
Reduction modulo $t^p-\zeta$ produces exactly the carry factors
$\zeta^{\eps(r,q)}$.

When $\zeta\ne1$, this quotient is a field by
Lemma~\ref{lem:field}. Since both polynomials have degree less
than $p$, one of them is identically zero. When $\zeta=1$
and $\xi\ne1$, reduce instead modulo $\Phi_p(t)$.
Again the quotient is a field, so one polynomial is a scalar
multiple of $\Phi_p(t)$ and has all coefficients equal.
At $(1,1)$ the $b_q$ are equal by~\eqref{eq:equalmass}.
It follows in every case that
\[
 (\zeta-1)a_rb_q=0,\qquad
 (a_r-a_{r'})(b_q-b_{q'})=0.
\]
Lemma~\ref{lem:interp} now gives~\eqref{eq:periodic}
and~\eqref{eq:rectangle}.

Finally, $y^{\eps(r,q)}U_{r,q}$ is a summand of~\eqref{eq:carry}.
All summands have nonnegative integer coefficients, while every
coefficient of $W$ is one. Thus the coefficients of $U_{r,q}$
are zero or one.
\end{proof}

\begin{lemma}[A binary matrix lemma]\label{lem:binary}
Suppose that a matrix $(u_{r,q})$ with entries in $\{0,1\}$ satisfies
\[
 u_{r,q}+u_{r',q'}=u_{r,q'}+u_{r',q}
\]
for all indices. Either each row is constant or each column is
constant. If $(u_{r,q})$ is $p\times p$ and
\[
 \sum_{r=0}^{p-1}u_{r,c-r}=1\qquad(c\in\Z_p),
\]
with column indices read modulo $p$, then it consists of a single
row of ones or a single column of ones, with zeros elsewhere.
\end{lemma}
\begin{proof}
The difference between two rows is constant. If two rows differ,
one is all zeros and the other all ones; comparison with the zero
row makes every row constant. Otherwise all rows are identical,
so each column is constant. The additional sums count the rows
or columns of ones and force that count to be one.
\end{proof}

\subsection{The common-complement construction}

\begin{theorem}\label{thm:carrier}
For each integer $\eta$, the group-ring element
\begin{equation}\label{eq:carrier}
 F_\eta(x,y)=\sum_{r=0}^{p-1}y^{\eta r}E_r(x,y)
\end{equation}
is the mask of a subset $D_\eta\subset G$, and
\begin{equation}\label{eq:common}
 D_\eta\oplus C_q=G\qquad(0\le q<p).
\end{equation}
The translated slices in~\eqref{eq:carrier} are therefore pairwise
disjoint. Moreover,
\begin{equation}\label{eq:external}
 D_\eta(x,1)=P(x,1),\qquad Q(x,1)=\sum_qC_q(x,1)
\end{equation}
in $\Z[x]/(x^n-1)$. These conclusions include $K=1$.
\end{theorem}

We call $D_\eta$ a \emph{carrier} of the slices $E_r$.

\begin{proof}
By~\eqref{eq:periodic}, the coefficients of $U_{r,q}$ are
constant in the second coordinate. Hence
\[
 U_{r,q}=H_L(y)\sum_{h=0}^{n-1}u_{r,q}(h)x^h,
 \qquad u_{r,q}(h)\in\{0,1\}.
\]
Since $yU_{r,q}=U_{r,q}$, the carry factors in~\eqref{eq:carry}
can be omitted, giving
$\sum_ru_{r,c-r}(h)=1$. Equation~\eqref{eq:rectangle} and
Lemma~\ref{lem:binary} show that, for each $h$, the matrix
$(u_{r,q}(h))$ has one row of ones or one column of ones.
Let $X_r$ be the set of $h$ for which the row has index $r$,
and $Y_q$ the set for which the column has index $q$. Thus
\begin{equation}\label{eq:partition}
 \Z_n=\left(\bigsqcup_rX_r\right)\sqcup
       \left(\bigsqcup_qY_q\right).
\end{equation}
Writing $f_r=[X_r]$ and $g_q=[Y_q]$ for their masks, we have
\begin{equation}\label{eq:rowcolumn}
 U_{r,q}=H_L(y)(f_r+g_q).
\end{equation}

We claim that all the $Y_q$ are empty. Evaluating
\eqref{eq:rowcolumn} at $(1,1)$ gives
\[
 |E_r|b=L(|X_r|+|Y_q|).
\]
Thus the $Y_q$ have equal cardinalities. Put
$X_*=\bigsqcup_rX_r$ and $Y_*=\bigsqcup_qY_q$; in particular,
$p\mid|Y_*|$.
Now set $e_r=E_r(x,1)$ and $c_q=C_q(x,1)$. Then
\begin{equation}\label{eq:projected}
 e_rc_q=L(f_r+g_q)
\end{equation}
in $\Z[\Z_n]$. Substitution into
$(e_rc_q)(e_{r'}c_{q'})=(e_rc_{q'})(e_{r'}c_q)$ yields
\begin{equation}\label{eq:rankidentity}
 (f_r-f_{r'})(g_{q'}-g_q)=0.
\end{equation}
Here the integer factor $L^2$ can be cancelled because
$\Z[\Z_n]$ is torsion-free as an abelian group.

For $N_{r,q}=f_rg_q$, equation~\eqref{eq:rankidentity} reads
$N_{r,q}=N_{r,0}+N_{0,q}-N_{0,0}$. Summing gives
\[
 [X_*][Y_*]
 =p\sum_rN_{r,0}+p\sum_qN_{0,q}-p^2N_{0,0},
\]
and hence
\begin{equation}\label{eq:XYzero}
 [X_*][Y_*]=0\quad\text{in }\F_p[\Z_n].
\end{equation}
On the other hand, $[X_*]+[Y_*]=H_n$ and
$H_n[Y_*]=|Y_*|H_n=0$ in this ring. Therefore $[Y_*]^2=0$.
Lemma~\ref{lem:reduced} implies $[Y_*]=0$ in $\F_p[\Z_n]$.
As the integer coefficients of $[Y_*]$ are zero or one,
this forces $Y_*=\varnothing$. We have proved
\begin{equation}\label{eq:rowonly}
 U_{r,q}=H_L(y)f_r,\qquad \Z_n=\bigsqcup_rX_r.
\end{equation}

It follows that
\[
 F_\eta C_q
 =\sum_ry^{\eta r}U_{r,q}
 =H_L(y)\sum_rf_r=W.
\]
The coefficients of $F_\eta$ are nonnegative integers, and
$C_q$ is nonempty. A coefficient of $F_\eta$ greater than one
would give a coefficient of $F_\eta C_q$ greater than one.
Thus $F_\eta$ is a set mask, proving~\eqref{eq:common} and the
disjointness assertion. Setting $y=1$ gives~\eqref{eq:external}.
For $K=1$ we have $L=1$ and $y=1$ in $R$; the same proof
applies, with no evaluation points having $\zeta\ne1$.
\end{proof}

\section{Cyclotomic factors under reduction}

\subsection{A transfer lemma}

Common complements allow us to pass a prime-power cyclotomic
divisor from a sum of masks to each summand.

\begin{lemma}\label{lem:transfer}
Let $u\ne p$ be primes, and suppose that
$D\oplus C_j=\Z_N$ for $0\le j<p$. If $a\ge1$, $\zeta$ has
order $u^a\mid N$, and
\begin{equation}\label{eq:aggregate}
 \sum_{j=0}^{p-1}\zeta^{t_j}C_j(\zeta)=0
\end{equation}
for some integers $t_j$, then $\Phi_{u^a}\mid C_j$ for every $j$.
\end{lemma}
\begin{proof}
All $C_j$ have cardinality $c=N/|D|$. By
Lemma~\ref{lem:census}, the set
\[
 S=\{v \in \Z_{\ge 1}: ~\Phi_{u^v}\mid C_j\}
\]
is independent of $j$ and has size $v_u(c)$.
Suppose that $a\notin S$. Write
\[
 H_*(x)=\prod_{v\in S}\Phi_{u^v}(x),\qquad
 C_j(x)=H_*(x)R_j(x),\quad R_j\in\Z[x].
\]
Then $H_*(\zeta)\ne0$ and
$R_j(1)=c/u^{v_u(c)}=:w$, with $u\nmid w$.
Cancel $H_*(\zeta)$ in~\eqref{eq:aggregate} and reduce modulo
$1-\zeta$. Since
\[
 \Z[\zeta]/(1-\zeta)
 \simeq\Z[x]/(\Phi_{u^a}(x),x-1)
 \simeq\F_u,
\]
every phase maps to $1$ and every $R_j(\zeta)$ maps to $w$.
This gives $pw=0$ in $\F_u$, contrary to $u\ne p$ and
$u\nmid w$. Thus $a\in S$.
\end{proof}

\subsection{Reduction to a slice}

\begin{proposition}\label{prop:layers}
Fix a prime $p$ and positive integers $n,K$ with $(n,p)=1$. 
Suppose that $A\oplus B=\Z_{np^K}$ and $\Phi_p\mid A$.
Let $C_0,\ldots,C_{p-1}$ be the lowest-digit slices of
$\kappa_K(A)$, and set
$L_j=\kappa_{K-1}^{-1}(C_j)\subset\Z_{np^{K-1}}$.
The $C_j$ have a common complement in
$\Z_n\times\Z_{p^{K-1}}$, and:
\begin{enumerate}
 \item If $u\ne p$ is prime, $u^b\mid n$ for a positive integer $b$, and $\Phi_{u^b}\mid A$,
 then $\Phi_{u^b}\mid L_j$ for every $j$.
 \item If $2\le\alpha\le K$ and $\Phi_{p^\alpha}\mid A$,
 then $\Phi_{p^{\alpha-1}}\mid L_j$ for every $j$.
 \item If $\xi^n=1$, $\ord\lambda=p^\alpha$, $1\le\alpha\le K$,
 and $\kappa_K(A)(\xi,\lambda)\ne0$, then
 $C_j(\xi,\lambda^p)\ne0$ for some $j$.
\end{enumerate}
\end{proposition}
\begin{proof}
Apply Theorem~\ref{thm:carrier} with
$P=\kappa_K(B)$, $Q=\kappa_K(A)$, and $\eta=0$ to obtain the
common complement. Abbreviating the bivariate mask of
$\kappa_K(A)$ to $A(x,t)$, we have
\begin{equation}\label{eq:assembly}
 A(x,t)=\sum_{j=0}^{p-1}t^jC_j(x,t^p).
\end{equation}

For (i), choose $\xi$ of order $u^b$. Equations~\eqref{eq:dictionary}
and~\eqref{eq:assembly} give
\[
 0=A(\xi,1)=\sum_jC_j(\xi,1)=\sum_jL_j(\xi).
\]
The transported slices $L_j$ also have a common complement,
so Lemma~\ref{lem:transfer}, with all $t_j=0$, applies.

For (ii), choose $\lambda$ of order $p^\alpha$ and put $\omega=\lambda^p$.
Then
\[
 0=A(1,\lambda)=\sum_j\lambda^jC_j(1,\omega).
\]
The coefficients belong to $\Q(\omega)$, and
$1,\lambda,\ldots,\lambda^{p-1}$ are linearly independent over
this field by Lemma~\ref{lem:field}. Hence
$C_j(1,\omega)=0$ for every $j$, as required.

Finally, if the left side of~\eqref{eq:assembly} is nonzero at
$(\xi,\lambda)$, one summand is nonzero. This proves (iii).
If $\ord\xi=d$, the corresponding root $\xi\lambda^p$ has
order $dp^{\alpha-1}$, including order $d$ when $\alpha=1$.
\end{proof}

\subsection{Reduction to a carrier}

\begin{proposition}\label{prop:carrierreduction}
Under the hypotheses of Theorem~\ref{thm:carrier}, let
$2\le a\le K$, $\ord\lambda=p^a$, and $\xi^n=1$.
Suppose that
\[
 P(1,\lambda)=0,\qquad P(\xi,\lambda)\ne0.
\]
For $\omega=\lambda^p$, there is an $\eta\in\{0,\ldots,p-1\}$
such that
\[
 D_\eta(1,\omega)=0,\qquad D_\eta(\xi,\omega)\ne0.
\]
Moreover, $D_\eta(x,1)=P(x,1)$ and
$D_\eta\oplus C_q=\Z_n\times\Z_{p^{K-1}}$ for every $q$.
\end{proposition}
\begin{proof}
The identity
\[
 0=P(1,\lambda)=\sum_r\lambda^rE_r(1,\omega)
\]
and Lemma~\ref{lem:field} imply $E_r(1,\omega)=0$ for every
$r$. Thus $D_\eta(1,\omega)=0$ for all $\eta$.

Put $e_r=E_r(\xi,\omega)$. The vector $(e_r)_{r=0}^{p-1}$
is nonzero, since $\sum_r\lambda^re_r=P(\xi,\lambda)\ne0$.
The polynomial $f(t)=\sum_re_rt^r$ has degree at most $p-1$,
whereas $1,\omega,\ldots,\omega^{p-1}$ are distinct because
$\ord\omega=p^{a-1}\ge p$. Hence
\[
 D_\eta(\xi,\omega)=f(\omega^\eta)\ne0
\]
for at least one $0\le\eta<p$. The remaining assertions are
part of Theorem~\ref{thm:carrier}.
\end{proof}

\section{Descent and the main theorem}\label{sec:descent}

\begin{proposition}\label{prop:onestep}
Suppose that $A\oplus B=\Z_M$ and that
$p,p_2,\ldots,p_k$ are distinct primes, with $k\ge2$.
Let $\alpha,\alpha_2,\ldots,\alpha_k$ be positive integers such that
\[
 p^\alpha\mid M,\quad p_i^{\alpha_i}\mid M,\qquad
 \Phi_{p^\alpha}\mid A,\quad\Phi_{p_i^{\alpha_i}}\mid A
 \quad(2\le i\le k).
\]
Put $d=\prod_{i=2}^kp_i^{\alpha_i}$. If $\Phi_{p^\alpha d}\nmid A$,
there is a factorization $A'\oplus B'=\Z_{M/p}$ with the following
properties:
\begin{enumerate}
 \item If $\alpha\ge2$, then $\Phi_{p^{\alpha-1}}$ and all
 $\Phi_{p_i^{\alpha_i}}$ divide $A'$, but
 $\Phi_{p^{\alpha-1}d}\nmid A'$.
 \item If $\alpha=1$, then all $\Phi_{p_i^{\alpha_i}}$ divide
 $A'$, but $\Phi_d\nmid A'$.
\end{enumerate}
\end{proposition}
\begin{proof}
Write $M=np^K$, with $(n,p)=1$, and choose roots of unity $\xi,\lambda$ of orders $d,p^\alpha$ respectively. Then
\[
 \kappa_K(A)(\xi,\lambda)=A(\xi\lambda)\ne0.
\]
By Lemma~\ref{lem:census}, $\Phi_p$ divides exactly one
of $A$ and $B$.

If $\Phi_p\mid A$, apply Proposition~\ref{prop:layers}.
Every transported slice retains the selected factors
$\Phi_{p_i^{\alpha_i}}$ and, when $\alpha\ge2$,
$\Phi_{p^{\alpha-1}}$. Choose a slice $A'=L_j$ that is
nonzero at $\xi\lambda^p$, and let $B'$ be its common
complement. The root $\xi\lambda^p$ has order
$dp^{\alpha-1}$, giving the required nondivisibility.

If $\Phi_p\mid B$, then $\alpha\ge2$. Apply
Proposition~\ref{prop:carrierreduction} with
$P=\kappa_K(A)$ and $Q=\kappa_K(B)$, and set
\[
 A'=\kappa_{K-1}^{-1}(D_\eta),\qquad
 B'=\kappa_{K-1}^{-1}(C_0).
\]
These sets tile $\Z_{M/p}$. The evaluations at
$(1,\lambda^p)$ and $(\xi,\lambda^p)$ give
$\Phi_{p^{\alpha-1}}\mid A'$ and
$\Phi_{dp^{\alpha-1}}\nmid A'$. For a primitive
$p_i^{\alpha_i}$th root $\xi_i$,
\[
 A'(\xi_i)=D_\eta(\xi_i,1)=P(\xi_i,1)=A(\xi_i)=0,
\]
so the other selected factors are retained.
\end{proof}

\medskip

\begin{proof}[\textbf{Proof of Theorem~\ref{thm:main}}]
Coven and Meyerowitz proved that (T1) and (T2) imply
tiling~\cite[Theorem A]{CM}, and that
\textup{(T1)} is necessary~\cite[Theorem B1]{CM}. 
It remains to prove that every finite integer tile satisfies \textup{(T2)}.

Suppose now that some integer tile fails (T2). Among all such
tiles and all choices of prime powers witnessing a failure,
choose $A$ and $(p_i^{\alpha_i})_{i=1}^k$ for which
$S=\alpha_1+\cdots+\alpha_k$ is minimal. Necessarily $k\ge2$.
Choose a cyclic realization $A\oplus B=\Z_M$ with each
$p_i^{\alpha_i}\mid M$. Evaluation at $M$th roots is unchanged
by reduction modulo $M$, so all the selected divisibilities
and the failure of $\Phi_{\prod_i p_i^{\alpha_i}}\mid A$
persist.

Apply Proposition~\ref{prop:onestep} with $p=p_1$.
If $\alpha_1\ge2$, it decreases $\alpha_1$ by one; if
$\alpha_1=1$, it removes $p_1$ from the list. In either case
we obtain a failure of (T2) with total exponent $S-1$.
Taking standard representatives in
$A'\oplus B'=\Z_N$ gives the integer tiling
$A'\oplus(B'+N\Z)=\Z$, contradicting minimality. 
Thus every integer tile satisfies (T2).
\end{proof}

\section{An example modulo 72}\label{sec:example}

Consider the sets from the final example of~\cite{CM}:
\begin{equation}\label{ex:sets}
\begin{aligned}
 P&=\{0,1,5,6,12,25,29,36,42,48,49,53\},\\
 Q&=\{0,8,16,18,26,34\}.
\end{aligned}
\end{equation}
We shall verify $P\oplus Q=\Z_{72}$ and compute the reductions
for $p=3$ and $p=2$. The ternary reduction illustrates the
common-complement construction when neither factor is
contained in a single residue class modulo $3$.

\subsection{Ternary slices}

Take $p=3$, $n=8$, and $K=2$. Under
$\kappa_2:\Z_{72}\to\Z_8\times\Z_9$, the slices are
\begin{equation}\label{ex:ternarysets}
\begin{aligned}
 E_0&=\{(0,0),(0,1),(2,2),(4,0),(4,1),(6,2)\},\\
 E_1&=\{1\}\times\Z_3,\qquad E_2=\{5\}\times\Z_3,\\
 C_0&=\{(0,0),(2,0)\},\qquad
 C_1=C_2=\{(0,2),(2,2)\}.
\end{aligned}
\end{equation}
Put $H=1+y+y^2$. In $\Z[x,y]/(x^8-1,y^3-1)$,
\begin{equation}\label{ex:ternarymasks}
\begin{aligned}
 E_0&=(1+x^4)(1+y)+(x^2+x^6)y^2,\\
 E_1&=xH,\qquad E_2=x^5H,\\
 C_0&=1+x^2,\qquad C_1=C_2=y^2C_0.
\end{aligned}
\end{equation}
The bivariate masks of the original factors are
\begin{equation}\label{ex:ternaryassembly}
\begin{aligned}
 P_3(x,t)&=(1+x^4)(1+t^3)+(x^2+x^6)t^6\\
 &\quad+x(t+t^4+t^7)+x^5(t^2+t^5+t^8),\\
 Q_3(x,t)&=(1+x^2)(1+t^7+t^8).
\end{aligned}
\end{equation}
In particular, $Q_3(1,\zeta_3)=0$, as required by
\eqref{eq:lowowner}.

Set
\[
 f_0=1+x^2+x^4+x^6,\qquad f_1=x+x^3,\qquad f_2=x^5+x^7.
\]
Direct multiplication gives
\begin{equation}\label{ex:Umatrix}
 (E_rC_q)_{0\le r,q\le2}
 =H\begin{pmatrix}
 f_0&f_0&f_0\\f_1&f_1&f_1\\f_2&f_2&f_2
 \end{pmatrix}.
\end{equation}
Since $yH=H$ and $\sum_rf_r=H_8(x)$, each carry sum
in~\eqref{eq:carry} equals $H_8(x)H$. Equivalently,
\[
 P_3(x,t)Q_3(x,t)=H_8(x)H_9(t)
 \quad\text{in }\Z[x,t]/(x^8-1,t^9-1),
\]
which proves $P\oplus Q=\Z_{72}$.
The partition in the proof of Theorem~\ref{thm:carrier} is
\[
 X_0=\{0,2,4,6\},\qquad X_1=\{1,3\},\qquad
 X_2=\{5,7\},\qquad Y_q=\varnothing.
\]

Here every carrier is the same:
\begin{equation}\label{ex:carrier3}
 F_\eta=E_0+y^\eta E_1+y^{2\eta}E_2
       =E_0+(x+x^5)H=:D^{(3)}(x,y).
\end{equation}
Thus $D^{(3)}=E_0\sqcup(\{1,5\}\times\Z_3)$ and
$D^{(3)}\oplus C_q=\Z_8\times\Z_3$ for every $q$. 
Evaluating the second variable at $1$ gives 
\begin{equation}\label{ex:projections3}
\begin{aligned}
 D^{(3)}(x,1)=P_3(x,1)&=(1+x^4)(x+1)(x+2),\\
 \sum_qC_q(x,1)=Q_3(x,1)&=3(1+x^2).
\end{aligned}
\end{equation}
The inverse CRT map $(h,j)\mapsto9h+16j\pmod{24}$ gives
\begin{equation}\label{ex:flat3}
\begin{aligned}
 D^{(3),\flat}&=\{0,1,2,4,5,9,12,13,14,16,17,21\},\\
 L_0&=\{0,18\},\qquad L_1=L_2=\{2,8\},
\end{aligned}
\end{equation}
with $D^{(3),\flat}\oplus L_j=\Z_{24}$.

\subsection{Cyclotomic divisors in the ternary reduction}

The prime-power factors of the original masks are
\begin{equation}\label{ex:pure}
 S_P=\{2,8,9\},\qquad S_Q=\{3,4\}.
\end{equation}
Indeed,~\eqref{ex:projections3} gives $\Phi_2,\Phi_8\mid P$
and $\Phi_4\mid Q$, while
\[
 P_3(1,t)=(2+t+t^2)(1+t^3+t^6),\qquad
 Q_3(1,t)=2(1+t^7+t^8)
\]
give $\Phi_9\mid P$ and $\Phi_3\mid Q$.
The values at one of these factors already multiply to
$|P|=12$ and $|Q|=6$, so no further prime-power factors
can occur.

For the reduction of $Q$, the slices in~\eqref{ex:flat3}
have masks
\begin{equation}\label{ex:Lfactor}
\begin{aligned}
 L_0(x)&=(1+x^2)\sum_{k=0}^8(-1)^kx^{2k},\\
 L_1(x)=L_2(x)&=(1+x^2)(x^6-x^4+x^2).
\end{aligned}
\end{equation}
Thus their common prime-power factor set is $\{4\}$:
the reduction removes $\Phi_3$ and preserves $\Phi_4$.
For a primitive $2^a$th root $\zeta$, $a\in\{1,2,3\}$,
Lemma~\ref{lem:transfer} says that
\[
 \sum_{j=0}^2\zeta^{t_j}L_j(\zeta)=0
\]
can hold only for $a=2$, when every summand is zero.
For $a=1$ or $3$, cancellation of $\Phi_4(\zeta)$ followed
by reduction modulo $1-\zeta$ would give $1+1+1=0$ in
$\F_2$.
To see the nonvanishing assertion of
Proposition~\ref{prop:layers}, take $\ord\lambda=9$ and
$\omega=\lambda^3$. Then
\[
 Q_3(-1,\lambda)=2\lambda^7(1+\lambda+\lambda^2)\ne0,
 \qquad
 (C_j(-1,\omega))_{j=0}^2=(2,2\omega^2,2\omega^2).
\]
Consequently $\Phi_{18}\nmid Q$ and $\Phi_6\nmid L_j$
for all $j$.

For the reduction of $P$, the lowest factor $\Phi_3$
divides its complement $Q$. Proposition~\ref{prop:carrierreduction}
applies with $\ord\lambda=9$, $\omega=\lambda^3$, and $\xi=i$.
We have $P_3(1,\lambda)=0$ and
\begin{equation}\label{ex:carriernonzero}
 E_0(i,\omega)=-4\omega^2,\qquad
 E_1(i,\omega)=E_2(i,\omega)=0.
\end{equation}
Thus $P_3(i,\lambda)=-4\omega^2\ne0$.
Every $\eta$ works in this example, since all carriers coincide:
\[
 D^{(3)}(1,\omega)=0,\qquad
 D^{(3)}(i,\omega)=-4\omega^2\ne0.
\]
In univariate notation,
\begin{equation}\label{ex:carrierdivisors}
 \Phi_3\mid D^{(3),\flat},\qquad
 \Phi_{36}\nmid P,\qquad
 \Phi_{12}\nmid D^{(3),\flat}.
\end{equation}
The projection identity preserves $\Phi_2$ and $\Phi_8$,
so $S_{D^{(3),\flat}}=\{2,3,8\}$.

\subsection{Binary reduction}

Now use $\Z_{72}\simeq\Z_9\times\Z_8$, with $p=2$ and
$K=3$. Since $\Phi_2\mid P$, take the lowest-digit slices
of $P$:
\begin{equation}\label{ex:binarysets}
\begin{aligned}
 S_0&=\{(0,0),(0,2),(3,0),(3,2),(6,1),(6,3)\},\\
 S_1&=\{(1,0),(4,0),(7,0),(2,2),(5,2),(8,2)\}.
\end{aligned}
\end{equation}
In $\Z[x,y]/(x^9-1,y^4-1)$, put
$J=1+x^3+x^6$ and $G_0=1+x^7+x^8$. Then
\begin{equation}\label{ex:binarymasks}
 S_0=(1+y^2)(1+x^3+x^6y),\qquad
 S_1=J(x+x^2y^2).
\end{equation}
The even slice of $Q$ is
\[
 D^{(2)}=\{0,7,8\}\times\{0,1\},\qquad
 D^{(2)}(x,y)=G_0(1+y),
\]
and its odd slice is empty. Since $G_0J=H_9(x)$, direct
multiplication gives
\begin{equation}\label{ex:binaryproducts}
 D^{(2)}S_0=D^{(2)}S_1=H_9(x)H_4(y).
\end{equation}
Thus $D^{(2)}$ is a common complement for $S_0,S_1$.
The inverse map $(h,j)\mapsto28h+9j\pmod{36}$ gives
\begin{equation}\label{ex:flat2}
\begin{aligned}
 D^{(2),\flat}&=\{0,8,9,16,17,25\},\\
 M_0&=\{0,12,15,18,30,33\},\\
 M_1&=\{2,4,14,16,26,28\},
\end{aligned}
\end{equation}
and $D^{(2),\flat}\oplus M_j=\Z_{36}$.

If $\xi$ is a primitive ninth root, then $J(\xi)=0$, so
$S_0(\xi,1)=S_1(\xi,1)=0$. Thus $\Phi_9$ is preserved.
For $\ord\lambda=8$ and $\omega=\lambda^2$,
\[
 S_0(1,\omega)=(1+\omega^2)(2+\omega)=0,\qquad
 S_1(1,\omega)=3(1+\omega^2)=0.
\]
This lowers $\Phi_8$ to $\Phi_4$, giving
\begin{equation}\label{ex:Mpure}
 S_{M_0}=S_{M_1}=\{4,9\}.
\end{equation}
In the transfer lemma with $p=2$ and $u=3$, the common
factor is $\Phi_9$ and $(M_j/\Phi_9)(1)=2$.
A phased sum vanishing at a primitive cube root would
therefore give $2+2=0$ in $\F_3$, which is impossible.

The nonvanishing condition can select just one slice.
Take $\xi$ to be a primitive cube root, while keeping
$\ord\lambda=8$ and $\omega=\lambda^2$. Then
\[
 S_0(\xi,\omega)=0,\qquad
 S_1(\xi,\omega)=3(\xi-\xi^2)\ne0.
\]
Writing $P_2$ for the mask of $P$ in these coordinates,
we obtain $P_2(\xi,\lambda)=3\lambda(\xi-\xi^2)\ne0$.
Hence
\begin{equation}\label{ex:binarynonzero}
 \Phi_{24}\nmid P,\qquad
 \Phi_{12}\mid M_0,\qquad \Phi_{12}\nmid M_1.
\end{equation}
Only $M_1$ retains this nonzero evaluation. The root $\xi$ here has order $3$, whereas the prime-power divisor
retained above is $\Phi_9$.

One can also reduce $Q$ through its carrier $D^{(2)}$.
Indeed,
\[
 Q_2(x,t)=G_0(x)(1+t^2),\qquad
 D^{(2)}(x,y)=G_0(x)(1+y).
\]
Thus $\Phi_4\mid Q$ descends to $\Phi_2\mid D^{(2),\flat}$,
and the projection preserves $\Phi_3$. This gives
$S_{D^{(2),\flat}}=\{2,3\}$.
This choice does not illustrate the nonvanishing conclusion
of Proposition~\ref{prop:carrierreduction}: when
$\ord\lambda=4$, $Q_2(\xi,\lambda)=0$ for every $\xi^9=1$;
when $\ord\lambda=8$, the pure-zero hypothesis fails.

\begin{center}
\small
\renewcommand{\arraystretch}{1.2}
\begin{tabular}{ccccc}
\toprule
Set reduced & $p$ & Factor divisible by $\Phi_p$ & Reduced set & Prime-power cyclotomic indices\\
\midrule
$P$ & $3$ & $Q$ & $D^{(3),\flat}\subset\Z_{24}$ & $\{2,3,8\}$\\
$Q$ & $3$ & $Q$ & $L_j\subset\Z_{24}$ & $\{4\}$\\
$P$ & $2$ & $P$ & $M_j\subset\Z_{36}$ & $\{4,9\}$\\
$Q$ & $2$ & $P$ & $D^{(2),\flat}\subset\Z_{36}$ & $\{2,3\}$\\
\bottomrule
\end{tabular}
\end{center}

\subsection{The case \texorpdfstring{$K=1$}{K=1}}

Continue from $L_0\oplus D^{(3),\flat}=\Z_{24}$ with
$p=3$, $n=8$, and $K=1$. The slices, now subsets of $\Z_8$,
are
\[
\begin{aligned}
 \widehat E_0&=\{0,2\},\qquad
 \widehat E_1=\widehat E_2=\varnothing,\\
 \widehat C_0&=\widehat C_1=\{0,1,4,5\},\qquad
 \widehat C_2=\{1,2,5,6\}.
\end{aligned}
\]
For each $q$,
$(1+x^2)\widehat C_q(x)=H_8(x)$ in $\Z[x]/(x^8-1)$.
Every carrier is therefore $\{0,2\}$, and
$\{0,2\}\oplus\widehat C_q=\Z_8$.

\section*{Acknowledgements}

The last four authors were classmates in the Chern class at
Nankai University. Since December 2024 they have met weekly
online to study integer tilings. They thank Terry Tao for
the notes~\cite{Tao} that introduced them to this subject.

SageMath, MATLAB, and language models including DeepSeek and
Doubao were used to write code for small-case checks. 
The general reduction emerged from extended work
with ChatGPT-5.6 Sol, with the participation of the first two 
authors, graduate students at the ZJNU Tiling Studio
under the supervision of the 
corresponding author. All authors independently verified the derivations, 
computations, logical dependencies and numerous examples, and 
completed the proof together. They take responsibility for its mathematical
content and have retained the interaction records.

\end{document}